\documentclass[12pt]{article}
\usepackage{amssymb,latexsym}
\usepackage[intlimits]{amsmath}
\usepackage{hyperref}
\usepackage{amsthm}
\usepackage{accents}
\usepackage{cite}
\usepackage{graphicx}
\usepackage{color}
\usepackage{cleveref}

\usepackage[normalem]{ulem}
\numberwithin{equation}{section}

\newtheorem{theorem}{Theorem}[section]
\newtheorem{definition}[theorem]{Definition}
\newtheorem{proposition}[theorem]{Proposition}
\newtheorem{lemma}[theorem]{Lemma}

\newtheorem{remark}[theorem]{Remark}
\newtheorem{remarkanddef}[theorem]{Remark and Definition}
\newtheorem{corollary}[theorem]{Corollary}

\newcommand\R{{\mathbb R}}
\newcommand\N{{\mathbb N}}
\newcommand\sphere{{\mathbb S}^{n-1}}
\newcommand\capa{\operatorname{Cap}(\Omega)}
\newcommand\sarea{\left|\mathbb{S}^{n-1}\right|}
\newcommand\crit{\operatorname{Crit}(u)}
\newcommand{\definedas}{\mathrel{\raise.095ex\hbox{\rm :}\mkern-5.2mu=}}
\newcommand{\asdefined}{\mathrel{=\mkern-5.2mu}\raise.095ex\hbox{\rm :}\;}
\newcommand{\diver}{\operatorname{div}}
\newcommand{\supp}{\operatorname{supp}}
\newcommand{\critDu}{\operatorname{Crit}(\vert Du\vert^{2})}

\title{A new proof of the Willmore inequality\\ via a divergence inequality}
\author{Carla Cederbaum\thanks{cederbaum@math.uni-tuebingen.de}\; and Anabel Miehe \\Mathematics Department \\Eberhard Karls Universit\"at T\"ubingen}
\date{}

\begin{document}
\maketitle

\begin{abstract} 
We present a new proof of the Willmore inequality for an arbitrary bounded domain $\Omega\subset\R^{n}$ with smooth boundary. Our proof is based on a parametric geometric inequality involving the electrostatic potential for the domain~$\Omega$; this geometric inequality is derived from a geometric differential inequality in divergence form. Our parametric geometric inequality also allows us to give new proofs of the quantitative Willmore-type and the weighted Minkowski inequalities by Agostiniani and Mazzieri.
\end{abstract}
	
\section{Introduction and results}\label{section intro}
We consider a bounded domain $\Omega\subset \R^n$ for $n\geq 3$ with smooth boundary $\partial\Omega$, with smooth \emph{(electrostatic) potential} $u\colon\R^n\setminus\Omega \to \R$ solving 
\begin{align}\label{id-problem}
\begin{cases}
	\Delta u= 0 &\text{in } \R^n\setminus\overline{\Omega},\\
	u=1 &\text{on } \partial\Omega,\\
	u(x)\to 0 &\text{as } |x|\to\infty,
\end{cases}
\end{align}
where $\Delta$ denotes the (Euclidean) Laplacian. It is well-known that every domain $\Omega$ as above carries a unique electrostatic potential $u$, see \Cref{section preliminaries}.

The main goal of this work is to prove the following parametric geometric inequality which can be exploited to prove the Willmore inequality as well as the quantitative Willmore-type and the weighted Minkowski inequalities by Agostiniani and Mazzieri~\cite{Agostiniani.2020}, see \Cref{willmore cor,quant will cor,weighted mink cor}. See \Cref{section relation} for a brief account of related results and \Cref{section discussion} for a more detailed comparison of our method with the monotonicity formula approach based on the electrostatic potential taken in \cite{Agostiniani.2020}.

\begin{theorem}[Parametric geometric inequality]\label{geom inequ thm}
Let $n\geq3$ and consider a parameter $\beta\in\R$ satisfying $\beta\geq \frac{n-2}{n-1}$. Let $\Omega\subset\R^{n}$ be a bounded domain with smooth boundary~$\partial\Omega$. Let $u$ be the electrostatic potential for $\Omega$, i.e., the unique smooth solution to \eqref{id-problem}, and consider parameters $c,d\in \R$ satisfying $c+d\geq 0$ and $d\geq0$. One then has
\begin{align}\label{geom inequ}
\begin{split} 
& d(n-2)^{\beta+1}\sarea\capa^{\frac{n-2-\beta}{n-2}} \\
 & \quad\leq \beta (c+d)\int_{\partial\Omega} |Du|^{\beta}Hd\sigma+\left[d-\frac{(n-1)\beta(c+d)}{n-2}\right]\int_{\partial\Omega} |Du|^{\beta+1}\,d\sigma,
\end{split}
\end{align} 
where $Du$ denotes the (Euclidean) gradient of $u$, $H$ denotes the (Euclidean) mean curvature of $\partial\Omega$ with respect to the (Euclidean) unit normal $\nu$ pointing towards infinity, $d\sigma$ denotes the (Euclidean) area element induced on $\partial\Omega$, $\sarea$ denotes the surface area of the unit sphere, and $\capa$ denotes the electrostatic capacity of $\Omega$ as defined in \Cref{capacity}. Moreover, equality holds in \eqref{geom inequ} if and only if $\Omega$ is a round ball (unless $c=d=0$).
\end{theorem}

Choosing $\beta=n-2$, we obtain the classical Willmore inequality \eqref{willmore} from \eqref{geom inequ} as follows: We first choose $c=1$, $d=0$ in \eqref{geom inequ} from which we find 
\begin{align}\label{eq:A}
\int_{\partial\Omega}\vert Du\vert^{n-1}\,d\sigma&\leq (n-2)^{n-1}\int_{\partial\Omega}\left\vert \frac{H}{n-1}\right\vert^{n-1}d\sigma
\end{align}
by applying the H\"older inequality with $p=n-1$, $p^{*}=\frac{n-1}{n-2}$ to the right hand side and performing some simple algebraic manipulations. On the other hand, choosing $c=-1$, $d=1$ in \eqref{geom inequ} gives 
\begin{align}\label{eq:B}
(n-2)^{n-1}\sarea&\leq \int_{\partial\Omega}\vert Du\vert^{n-1}\,d\sigma,
\end{align}
which proves the Willmore inequality when combined with \eqref{eq:A}. As the equality assertions for \eqref{eq:A}, \eqref{eq:B} (i.e., those of \Cref{geom inequ thm}) are stronger than those of the H\"older inequality, the equality claim readily follows.
 \begin{corollary}[$(n-1)$-dimensional Willmore inequality {\cite{Agostiniani.2020,Chen.1971,Chen.1971b}}]\label{willmore cor}
Let $n\geq3$ and let $\Omega\subset \R^n$ be a bounded domain with smooth boundary. Then 
\begin{align}\label{willmore}
   \sarea \leq \int_{\partial\Omega}\left(\frac{\vert H\vert}{n-1}\right)^{n-1}d\sigma.
\end{align}
Moreover, equality holds in \eqref{willmore} if and only if $\Omega$ is a round ball.
\end{corollary}

The same argument for arbitrary $\beta\geq\frac{n-2}{n-1}$ and $p\definedas\beta+1$ gives the following generalization of the classical Willmore inequality. 
\begin{corollary}[Generalized Willmore inequality {\cite[Corollary 4.4]{Agostiniani.2020}}]\label{willmore cor p}
Let $n\geq3$, $p\geq\frac{2n-3}{n-1}$, and let $\Omega\subset \R^n$ be a bounded domain with smooth boundary. Then 
\begin{align}\label{eq:Minkowskip}
\sarea \capa^{\frac{n-1-p}{n-2}}&\leq \int_{\partial\Omega}\left(\frac{\vert H\vert}{n-1}\right)^{p}d\sigma.
\end{align}
Moreover, equality holds in \eqref{willmore} if and only if $\Omega$ is a round ball.
\end{corollary}

In addition to providing a new proof for the classical Willmore inequality, \Cref{geom inequ thm} also readily provides a new proof of the weighted Minkowski and the quantitative Willmore-type inequalities introduced by Agostiniani and Mazzieri. Specifically, choosing $\beta=\frac{n-2}{n-1}$, $c=-1$, $d=1$ in \Cref{geom inequ thm} and using \Cref{capacity in terms of potential} relating $\capa$ to $\int_{\partial\Omega}\vert Du\vert\,d\sigma$ gives the following weighted Minkowski inequality.

\begin{corollary}[Weighted Minkowski inequality, {\cite[Theorem 1.3]{Agostiniani.2020}}]\label{weighted mink cor}
Let $n\geq3$ and let $\Omega\subset \R^n$ be a bounded domain with smooth boundary and electrostatic potential $u$. Then
 \begin{align}\label{weigh mink}
 \begin{split}
 &\frac{n-2}{n-1}\int_{\partial\Omega}\left|Du\right|^{\frac{n-2}{n-1}}\left[H-\left(\frac{n-1}{n-2}\right)\left|Du\right|\right]d\sigma\\
 &\quad\leq (n-2)^{\frac{2n-3}{n-1}}\sarea^{\frac{n-2}{n-1}}\left(\frac{\capa}{|\partial\Omega|} \right)^{\frac{n-2}{n-1}}\!\left[\,\,\int_{\partial \Omega}\frac{H}{n-1}\,d\bar{\sigma}-\sarea ^{\frac{1}{n-1}}|\partial\Omega|^{\frac{n-2}{n-1}} \right]\!\!,
\end{split}    	
\end{align}
where
\begin{equation*}
  d\bar{\sigma}\definedas\left(\frac{|Du|}{\frac{1}{|\partial\Omega|}\int_{\partial \Omega}|Du|\,d\sigma} \right)^{\frac{n-2}{n-1}}d\sigma
  \end{equation*}
is a weighted area measure on $\partial\Omega$ and $\vert\partial\Omega\vert$ denotes the area of $\partial\Omega$ with respect to $d\sigma$. Moreover, equality holds in \eqref{weigh mink} if and only if $\Omega$ is a round ball.
 \end{corollary}
 
Applying in addition the H\"older inequality for $p=n-1$, $p^{*}=\frac{n-1}{n-2}$ to the first term in the square bracket in \eqref{weigh mink}, we obtain the quantitative Willmore-type inequality~\eqref{quant will}. The equality case can be handled as above.
\begin{corollary}[Quantitative Willmore-type inequality, {\cite[Theorem 1.2]{Agostiniani.2020}}]\label{quant will cor}
Let $n\geq3$ and let $\Omega\subset \R^n$ be a bounded domain with smooth boundary and potential $u$. Then
 \begin{align}\label{quant will}
 \begin{split}
 &\frac{n-2}{n-1}\int_{\partial\Omega}|Du|^{\frac{n-2}{n-1}}\left[H-\left(\frac{n-1}{n-2}\right)|Du|\right]d\sigma\\
 &\quad\leq (n-2)^{\frac{2n-3}{n-1}}\sarea ^{\frac{n-2}{n-1}}\capa^{\frac{n-2}{n-1}}\!\left[\!\left(\;\int_{\partial\Omega}\left(\frac{\vert H\vert}{n-1}\right)^{n-1}\!d\sigma\!\right)^{\frac{1}{n-1}}\!\!\!-\sarea ^{\frac{1}{n-1}}\right]\!\!.
 \end{split}
 \end{align}
Moreover, equality holds in \eqref{quant will} if and only if $\Omega$ is a round ball.
\end{corollary}

Last but not least, note that it is no accident that we only use the choices $c=1$, $d=0$ and $c=-1$, $d=1$ when applying \eqref{geom inequ}. In fact, \eqref{geom inequ} is linear in $(c,d)\in\mathcal{P}\subset\R^{2}$ for $\mathcal{P}\definedas\{c+d\geq0,\; d\geq0\}$, where $\mathcal{P}\subset\R^{2}$ is characterized by linear inequalities, see also \Cref{cd graphic}. Hence \eqref{geom inequ} holds for all $(c,d)\in\mathcal{P}$ if and only if it holds for all $(c,d)\in\partial\mathcal{P}$ or if and only if it holds for $c=-1$, $d=1$ and for $c=1$, $d=0$.

\begin{figure}[h]
\centering
\includegraphics[keepaspectratio, width=12cm]{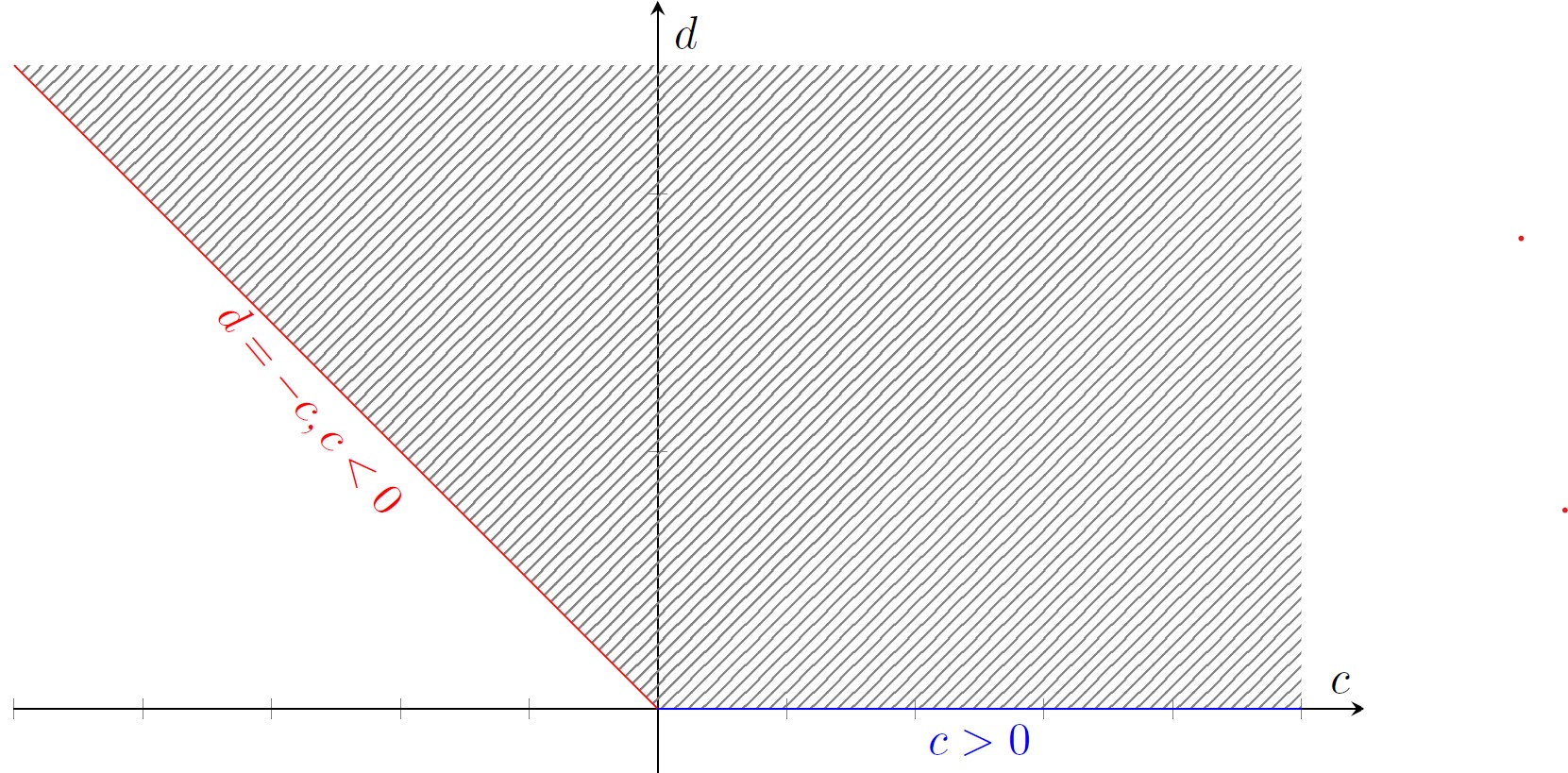}
\caption{The domain $\mathcal{P}\subset\R^2$ of all pairs $(c,d)$ for which \eqref{geom inequ} holds.}\label{cd graphic}
\end{figure}

Thus, \Cref{geom inequ thm} can equivalently be stated as follows.
\begin{theorem}[Geometric inequalities]\label{geom inequ thm 2}
Let $n\geq3$ and consider a parameter $\beta\in\R$ satisfying $\beta\geq \frac{n-2}{n-1}$. Let $\Omega\subset\R^{n}$ be a bounded domain with smooth boundary~$\partial\Omega$. Let $u$ be the electrostatic potential for $\Omega$, i.e., the unique smooth solution to \eqref{id-problem}. One then has     	
\begin{align}\label{geom inequ A}
(n-2)^{\beta+1}\sarea\capa^{\frac{n-2-\beta}{n-2}} &\leq \int_{\partial\Omega} |Du|^{\beta+1}\,d\sigma,\\\label{geom inequ B} 
\left(\frac{n-1}{n-2}\right)\int_{\partial\Omega} |Du|^{\beta+1}\,d\sigma  &\leq \int_{\partial\Omega} |Du|^{\beta}Hd\sigma.
\end{align} 
Moreover, equality holds in each of \eqref{geom inequ A} and \eqref{geom inequ B} if and only if $\Omega$ is a round ball.
\end{theorem}

\Cref{willmore cor,willmore cor p} then follow from \eqref{geom inequ A} \emph{and} \eqref{geom inequ B}, while \Cref{weighted mink cor,quant will cor} follow from just \eqref{geom inequ A}. We would like to point out that \eqref{geom inequ A} and \eqref{geom inequ B} appear as (4.2) and (4.1) in \cite{Agostiniani.2020}, respectively; i.e., they also follow from Agostiniani and Mazzieri's monotonicity formula approach, see \Cref{section discussion} for more information.

\vfill

\subsection{Strategy of proof of \Cref{geom inequ thm}}
To prove the parametric geometric inequality \eqref{geom inequ} of \Cref{geom inequ thm}, we will first show the following geometric differential inequality which takes a divergence form. 

\begin{theorem}[Divergence inequality] \label{divId Thm}
Let $n\geq3$ and let $\Omega\subset\R^{n}$ be a bounded domain with smooth boundary. Let $u$ be the electrostatic potential for $\Omega$, i.e., the unique smooth solution to \eqref{id-problem}, and let $\crit$ denote the set of critical points of $u$. Let $\beta\geq0$ and set
\begin{equation}\label{abeta}
a_{\beta}\definedas\frac{\beta}{4}\left(\beta-\frac{n-2}{n-1}\right).
\end{equation}
Then the \emph{divergence inequality}
\begin{align}\label{RobDivId5}
\begin{split}
&\diver\!\left(F(u)D|Du|^{\beta}+G(u)|Du|^{\beta} Du\right) \\\
&\quad\geq  a_{\beta}\, F(u) \left|D|Du|^2 - \frac{2(n-1)|Du|^2Du}{(n-2)u}\right|^2 |Du|^{\beta-2}
\end{split}
\end{align}
holds on $\R^n\setminus(\overline{\Omega}\cup \crit)$ for the smooth functions $F,G\colon(0,1]\to\R$ given by
\begin{align}\label{solution F}
F(u)&\definedas (cu+d) u^{1-\frac{n-1}{n-2}\beta},\\ \label{solution G}
 G(u)&\definedas-\frac{(n-1)\beta}{(n-2)u}F(u)+du^{-\frac{n-1}{n-2}\beta}
 \end{align}
 for any fixed constants $c,d\in\R$ satisfying $c+d\geq0$, $d\geq0$. Here, $\diver$ denotes the (Euclidean) divergence. Moreover, if $\beta\geq \frac{n-2}{n-1}$, 
 \begin{align}\label{divPos}
|Du|^2\,\diver\!\left(F(u)D|Du|^{\beta}+G(u)|Du|^{\beta} Du\right)\geq 0
 \end{align}
holds on $\R^n\setminus(\overline{\Omega}\cup \crit)$ with equality if and only if $\Omega$ is a round ball (unless $c=d=0$).
\end{theorem}

\Cref{geom inequ thm} then follows from \Cref{divId Thm} via the divergence theorem, appealing to a rigorous analysis of the behavior of $u$ near $\crit$, see \Cref{section geometric results}. 

\Cref{RobDivId5} vaguely resembles a geometric differential identity for the ``static potential'' of an asymptotically flat static vacuum system in general relativity first derived for $n=3$ and $\beta=2$ by Robinson~\cite{Robinson.1977} and then generalized to $n\geq3$, $\beta\geq2$ by Cederbaum, Cogo, Leandro, and Paulo dos Santos~\cite[Theorem 1.3]{Cederbaum.b}.  As in \cite{Cederbaum.b}, \Cref{divId Thm} also allows to derive monotonicity formulas for certain functionals arising as surface integrals on level sets of the electrostatic potential $u$, namely for those corresponding to the divergence on the left hand side of \eqref{RobDivId5} via the divergence theorem. These turn out to be closely related to the monotone functionals that Agostiniani and Mazzieri~\cite{Agostiniani.2020} use in their monotonicity formula approach to proving \eqref{willmore}--\eqref{geom inequ B} . This will be discussed further in \Cref{section discussion}.

\subsection{Relation to previous results}\label{section relation}
The $2$-dimensional Willmore inequality 
\begin{equation*}
16\pi\leq \int_{\partial\Omega} H^2\, d\sigma
\end{equation*}
was first shown by Willmore in 1968 in $3$-dimensional Euclidean space \cite[Theorem 1]{Willmore.1968}. The Willmore inequality was generalized to higher dimensions $n\geq3$ by Chen, see \cite{Chen.1971,Chen.1971b} and the references given therein. Another proof was presented by Agostiniani and Mazzieri in \cite{Agostiniani.2020} using a monotonicity formula approach in the potential theory setup described above. Versions of the Willmore inequality are also known in other ambient spaces, see e.g.~\cite{Agostiniani.2017,Hu.2018,Viana.22.07.2019,Agostiniani.2020b,Schulze.2020} and the references given therein. Our approach gives a new proof of the Willmore inequality in Euclidean space in $n\geq3$ dimensions, based on potential theory and a divergence inequality. It is conceivable that our approach may also be used to prove versions of the Willmore inequality and related inequalities in other ambient spaces and other geometric inequalities that are amenable to linear or even non-linear potential theory; this is work in progress by the first named author and collaborators. It may also be relevant for studying the stability of such inequalities, see e.g.~\cite{Chodosh.06.06.2023}.

As \eqref{eq:Minkowskip}--\eqref{geom inequ B}  are not of central interest here, we refer the interested reader to \cite[Section 1]{Agostiniani.2020} for a discussion of further results, references, and techniques of proof related to these inequalities.

\paragraph*{This paper is structured as follows:} In \Cref{section preliminaries}, we fix our notation and recall some well-known statements on harmonic functions and related concepts. In \Cref{section divergence inequality}, we prove the divergence inequality \eqref{RobDivId5} stated in \Cref{divId Thm}. In \Cref{section geometric results}, we prove \Cref{geom inequ thm}  (or equivalently \Cref{geom inequ thm 2}). Last but not least, in \Cref{section discussion}, we compare our approach based on the geometric differential inequality \eqref{divId Thm} to the monotonicity formula approach taken in \cite{Agostiniani.2020} and in particular derive the above-mentioned monotone functionals and compare them to those obtained in~\cite[Theorem 1.1]{Agostiniani.2020}.

\paragraph{Acknowledgements.} The work of the first named author is supported by the focus program on Geometry at Infinity (Deutsche Forschungsgemeinschaft,  SPP 2026). We would like to thank Virginia Agostiniani for interesting scientific discussions. This manuscript has no associated data.

\section{Preliminaries}\label{section preliminaries}
As implicitly stated in \Cref{section intro}, we will treat $\R^{n}$ as a smooth manifold carrying the Euclidean metric $\delta$. The sign of the Laplacian $\Delta u$ of a smooth function $u$ defined on the \emph{exterior region} $\R^{n}\setminus\Omega$ of a bounded domain $\Omega$ with smooth boundary $\partial\Omega$ is such that $\Delta=\sum_{i=1}^{n}\frac{\partial^{2}}{(\partial{x^{i}})^{2}}$ in Cartesian coordinates $(x^{i})$. The mean curvature $H$ of $\partial\Omega$ is computed with respect to the unit normal $\nu$ pointing towards infinity, with our sign and scaling convention being such that the unit round sphere $\sphere$ has mean curvature $H=n-1$ with respect to the unit normal $\nu$ pointing towards infinity.

Now let $\Omega\subset\R^{n}$ be a bounded domain with smooth boundary $\partial\Omega$. By standard elliptic theory\footnote{One has to be slightly careful as the exterior region $\R^{n}\setminus\Omega$ is unbounded. However, this is easy to work around by approximating $\R^{n}\setminus\Omega$ by an exhaustion $\left\{\left(\R^{n}\setminus\Omega\right)\cap B_{R}\right\}_{R>R_{0}}$ giving smooth solutions $u_{R}$ of the restricted problem which converge to a solution $u$ of \eqref{id-problem} in $C^{k}_{\text{loc}}$ for all $k\in\N$. Here, $B_{R}\subset\R^{n}$ denotes the open ball of radius $R>0$ centered at the origin. For the maximum principle, one argues by contradiction and can then restrict to a compact domain $U$ with smooth boundary $\partial U$ around potential interior maxima/minima. For the Hopf lemma, one restricts to a suitable compact domain $V$ with smooth boundary enclosing $\partial\Omega$, say with $u=\kappa$ on $\partial V\setminus\partial\Omega$ for some suitable $0<\kappa<1$, applying the maximum principle.\label{foot:noncompact}} (see e.g. \cite[Chapters 2, 3]{Gilbarg.2001}), there exists a unique smooth solution $u\colon\R^{n}\setminus\Omega\to\R$ to \eqref{id-problem}, called the \emph{(electrostatic) potential (of $\Omega$)}. The maximum principle and the Hopf lemma\textsuperscript{\ref{foot:noncompact}} inform us that $0<u<1$ in $\R^{n}\setminus\overline{\Omega}$ and that the normal derivative $\nu(u)$ of $u$ satisfies $\nu(u)<0$ on $\partial\Omega$. In particular, $\partial\Omega$ is a regular level set of the potential $u$. Moreover, for any regular level set $\{u=u_{0}\}$ of $u$, including $\partial\Omega=\{u=1\}$, one finds 
\begin{equation}\label{eq:nu}
\nu=-\frac{Du}{\vert Du\vert}
\end{equation}
for the unit normal $\nu$ pointing towards infinity. Furthermore, $u$ is real analytic on $\R^{n}\setminus\Omega$ (see e.g.~\cite[Theorem 10]{Evans.2022}). This also implies that $\vert Du\vert^{2}$ is real analytic on $\R^{n}\setminus\Omega$. For later convenience, let us state here that the Bochner identity for the harmonic function $u$ reduces to 
\begin{equation}\label{eq:Bochner}
\Delta |D u|^2=2|D^2u|^2.
\end{equation}

For convenience, let us state that for the round open ball $\Omega=B_{R}(z)\subset\R^{n}$ of radius $R>0$ and center $z\in\R^{n}$, one finds the electrostatic potential
\begin{equation}\label{radial u}
u_{R,z}(x)=\left(\frac{R}{\vert x-z\vert}\right)^{n-2}
\end{equation}
for all $x\in\R^{n}\setminus B_{R}(z)$. It can easily be checked that $u_{R,z}$, $\partial\Omega=\sphere_{R}(z)$ saturate \eqref{geom inequ} and all other inequalities in \Cref{section intro}.
 
An important concept related to problem \eqref{id-problem} is the (electrostatic) capacity which is defined in the following way.
\begin{definition}[Electrostatic capacity, see e.g. {\cite[Chapter 2.9]{Gilbarg.2001}}]\label{capacity}
Let $n\geq3$ and let $\Omega\subset \R^n$ be a bounded domain with smooth boundary. The quantity
\begin{equation}
\capa\definedas\inf \left\{\frac{1}{(n-2)\left| \mathbb{S}^{n-1}\right|}\int_{\R^n}|D\omega|^2 \,d\mu\:\left\vert\right.\:\omega\in C^{\infty}_c(\R^n),\, \omega=1 \textit{ in }\Omega\right\}
\end{equation}
is called the \emph{(electrostatic) capacity of $\Omega$}. Here, $d\mu$ is the Lebesgue measure on $\R^{n}$.
\end{definition}

For the proof of the parametric geometric inequality \eqref{geom inequ} in \Cref{geom inequ thm}, we will need to exploit  the asymptotic behaviour of $u$ and its derivatives at infinity. This is well-understood; we collect all useful facts in the following theorem. 

\begin{theorem}[Asymptotics of $u$, see e.g {\cite{Kellogg.1967}}]\label{asymptotics}
Let $n\geq3$ and let $\Omega\subset \R^n$ be a bounded domain with smooth boundary and electrostatic potential $u$. Then
\begin{align*}
u&=\frac{\capa}{|x|^{n-2}}+o\left(|x|^{-(n-2)}\right),\\
D_{i}u&=-\frac{(n-2)\capa}{|x|^{n}}\,x_{i}+o\left(|x|^{-(n-1)}\right),\\
D_{i}D_{j}u&=\frac{(n-2)\capa}{|x|^{n+2}}\left(nx_{i}x_{j}-|x|^2\delta_{i j}\right)+o\left(|x|^{-n}\right)
\end{align*} 
for $i,j=1,\dots,n$ as $\vert x\vert\to\infty$.
\end{theorem} 

From \Cref{asymptotics} and \eqref{radial u}, one directly sees that $\operatorname{Cap}(B_{R}(z))=R^{n-2}$. We can hence reformulate \Cref{asymptotics} as stating that
\begin{align}
\begin{split}\label{eq:asymptoticsR}
u&=u_{R_{\Omega},0}+o_{2}\left(|x|^{-(n-2)}\right),\\
R_{\Omega}&\definedas\capa^{\frac{1}{n-2}}
\end{split}
\end{align}
as $\vert x\vert\to\infty$. Moreover, using the divergence theorem, \eqref{id-problem}, \eqref{eq:nu}, and \Cref{asymptotics}, the following well-known fact is immediate.
\begin{proposition}[Capacity in terms of potential]\label{capacity in terms of potential}
Let $n\geq3$ and let $\Omega\subset \R^n$ be a bounded domain with smooth boundary and electrostatic potential $u$. Then
\begin{equation}
(n-2)\sarea\capa =\int_{\partial\Omega} |Du|\,d\sigma.
\end{equation}
In particular, $\capa>0$.
\end{proposition}

\begin{remarkanddef}[Neighborhood of infinity]\label{rem:infty}
As $\capa\neq0$, $u$ regularly foliates some \emph{neighborhood $\mathcal{N}\subset\R^{n}$ of infinity}, i.e., there exists an open set $\mathcal{N}\supset B_{R}$ for some $R>0$ such that $u$ has no critical points in $\mathcal{N}$. 
\end{remarkanddef}

We will also need the following consequences of \Cref{asymptotics} which can be derived from \Cref{asymptotics} and \eqref{eq:asymptoticsR} by standard computations and arguments (see \cite[Section 3]{Anabel} or see e.g.~\cite[Lemma 2.2]{Cederbaum.b} for very similar detailed computations and arguments) and which relies on \Cref{rem:infty}. 
\begin{proposition}[Asymptotics]\label{limits of FGH}
Let $n\geq3$ and let $\Omega\subset \R^n$ be a bounded domain with smooth boundary and electrostatic potential $u$. Let $F$ and $G$ be as in \eqref{solution F}, \eqref{solution G}, respectively, for some $c,d\in\R$, and let $\beta\geq0$. Then
\begin{align*}
|Du|^{\beta}&=\vert Du_{R_{\Omega},0}\vert^{\beta}+o\left(|x|^{-(n-1)\beta}\right)=\frac{(n-2)^{\beta}\capa^{\beta}}{|x|^{\beta(n-1)}}+o\left(|x|^{-(n-1)\beta}\right),\\
F(u)&=F(u_{R_{\Omega},0})+o\left(\frac{1}{|x|^{n-2-(n-1)\beta}} \right)=\frac{d\capa^{1-\frac{(n-1)\beta}{n-2}}}{|x|^{n-2-(n-1)\beta}}+o\left(\frac{1}{|x|^{n-2-(n-1)\beta}} \right),\\
G(u)&=G(u_{R_{\Omega},0})(1+o(1))=d\capa^{-\frac{(n-1)\beta}{n-2}}\left[1-\frac{(n-1)\beta}{n-2} \right]|x|^{(n-1)\beta}\,(1+o(1))
\end{align*}
as $\vert x\vert\to\infty$. For the regular level sets of $u$ contained in a suitable neighborhood $\mathcal{N}$ of infinity, the area measure $d\sigma$ and mean curvature $H$ with respect to the unit normal $\nu$ from \eqref{eq:nu} asymptotically behave as
\begin{align}\label{asydsigma}
d\sigma&=(1+o(1))\,d\sigma_{\mathbb{S}^{n-1}_{R}},\\\label{asyH}
H&=\frac{n-1}{|x|}(1+o(1))
\end{align}
as $\vert x\vert\to\infty$.
\end{proposition}

We will also make use of the following fact which follows from straightforward computations (see \cite[Remark 4.3]{Anabel} for details).
\begin{lemma}[ODEs for $F$ and $G$]\label{ODEs}
Let $F$ and $G$ be as in \eqref{solution F}, \eqref{solution G}, respectively, for some $c,d\in\R$, and let $\beta>0$. Then $F$ and $G$ solve the coupled ODEs
\begin{align}
F'(u)+G(u)&=-\frac{8(n-1)\,a_{\beta}}{(n-2)\beta u}\,F(u),\\
G'(u)&=\frac{4(n-1)^2\, a_{\beta}}{(n-2)^2u^2}\,F(u)
\end{align}
 for $a_{\beta}$ as in \eqref{abeta}. For $\beta=0$, we instead have
\begin{align}
F'(u)+G(u)&=\frac{2}{u}\,F(u),\\\label{G'}
G'(u)&=0.
\end{align}
\end{lemma}

Next, we collect some useful and well-known properties of the electrostatic potential $u$ evaluated on regular level sets of $u$. These can be verified by direct but lengthy computations (see e.g.~\cite[Section 2.2]{Agostiniani.2020}, \cite[Section 3]{Anabel}).

\begin{lemma}[Properties of the potential]\label{Properties}
Let $n\geq3$ and let $\Omega\subset \R^n$ be a bounded domain with smooth boundary and electrostatic potential $u$. Let $\crit$ denote the set of critical points of $u$. Then for any $\beta\geq0$, the identities  
\begin{align}\label{p.14 id 1}
D|D u|^{\beta}&=\frac{\beta}{2} |Du|^{\beta-2}D|Du|^2,\\\label{p.14 id 2}
\Delta|Du|^{\beta}&=\frac{\beta}{2} |Du|^{\beta-2}\Delta|Du|^2+\beta(\beta-2)|Du|^{\beta-2}|D|Du||^2
\end{align}
hold on $\R^{n}\setminus\left(\overline{\Omega}\cap\crit\right)$. Moreover,
\begin{align}\label{norm rel hess grad norm}
\delta(D |Du|^2,\cdot)&=2D^2u(Du,\cdot)
\end{align}
holds on $\R^{n}\setminus\Omega$. Now consider a regular level set $\{u=u_{0}\}$ of $u$. The second fundamental form $h$ and mean curvature $H$ of $\{u=u_{0}\}$ with respect to the unit normal $\nu$ given by \eqref{eq:nu} satisfy
\begin{align}\label{2ndFF}
h&=-\left.\frac{D^2u}{|Du|}\right\vert_{T\{u=u_{0}\}\times T\{u=u_{0}\}},\\\label{mean curv eucl}
H&=\frac{D^2u(Du,Du)}{|Du|^3}.
\end{align}
\end{lemma}

For several estimates, it will be useful to apply the following inequality giving a lower bound on the norm of the Hessian of a harmonic function.
\begin{proposition}[Refined Kato inequality, {see\footnote{or see {\cite[Corollary 4.6]{Fogagnolo.2019}}  for $p=2$ or {\cite[Proposition 2.3]{Anabel}} for more detailed computations. The equality claim can be extracted from these proofs.} e.g.~\cite{Schoen}}]\label{kato prop}
Consider a Riemannian manifold $(M,g)$ and a $g$-harmonic function $f\colon M\to\R$. Then the \emph{refined Kato inequality} 
\begin{align}\label{kato}
|D^2_{\!g}f|^2_g\geq\frac{n}{n-1}|D_{\!g}|D_{\!g} f|_g|_g^2
\end{align}
holds on $M\setminus\operatorname{Crit}(f)$. Here, $\vert\cdot\vert_{g}$ denotes the tensor norm and $D_{\!g}$ denotes the covariant derivative induced by $g$. For $p\in M\setminus\operatorname{Crit}(f)$, let $W_{p}\subseteq\{f=f(p)\}$ be an open neighborhood of $p$ such that $Df\neq0$ on $W_{p}$. Then equality holds in \eqref{kato} at $p$ if and only if both
\begin{align*}
D^{2}_{g}f\vert_{p}(X,Y)&=\frac{\operatorname{tr}_{T}\left(D^{2}_{g} f\vert_{T_{p}W_{p}\times T_{p}W_{p}}\right)}{n-1}\,g_{p}(X,Y),\\
g_{p}(D_{g}\vert D_{g}f\vert_{g}\vert_{p},X)&=0
\end{align*}
hold for all $X,Y\in T_{p}W_{p}$. Here, $\operatorname{tr}_{T}$ denotes the trace with respect to the metric induced on $W_{p}$ by $g$.
\end{proposition}

To deal with the presence of critical points for small $\beta$ when applying the divergence theorem, we will take advantage of the following Morse--Sard theorem for real analytic functions.
\begin{theorem}[Morse--Sard theorem for real analytic functions, {\cite[Theorem 1]{Soucek1972}}]\label{morse sard}
Consider a real analytic function $f\colon D\to \R$ for an open set $D\subset\R^n$ and denote by $\operatorname{Crit}(f)$ the set of critical points of $f$. Then $f(\operatorname{Crit}(f)\cap K)$ is finite for all compact subsets $K\Subset D$.
\end{theorem}	

\begin{corollary}[$u(\crit)$ is finite]\label{coro:finite}
As a consequence of the Morse--Sard~\Cref{morse sard}, the fact that $\partial\Omega$ is a regular level set of $u$, continuity of $\vert Du\vert$, and \Cref{rem:infty}, $u(\crit)$ is finite.
\end{corollary}

\begin{remark}[On the size of $\crit$]\label{rem:sizecrit}
By the work of Cheeger--Naber--Valtorta~\cite{Cheeger} and Hardt--Hoffmann-Ostenhof--Hoffmann-Ostenhof--Nadirashvili~\cite{Hardt}, we know that the Hausdorff dimension of the critical set $\crit$ is bounded above by $n-2$ as $u$ is harmonic. Hence in particular $\crit$ is a set of Lebesgue measure zero or in other words $u$ is regular Lebesgue almost everywhere. Moreover, by \Cref{rem:infty}, $\crit$ is compact.
\end{remark}

\begin{corollary}[$\vert Du\vert^{2}(\critDu)$ is discrete]\label{coro:Dusquared}
As a consequence of the Morse--Sard~\Cref{morse sard} and the real analyticity of $\vert Du\vert^{2}$, $\vert Du\vert^{2}(\critDu)$ is discrete. Moreover, as $\capa>0$ and $\vert D\vert Du\vert^{2}\vert=\vert D\vert Du_{R_{\Omega},0}\vert^{2}\vert+o(\vert x\vert^{-2n+1})$ as $\vert x\vert\to\infty$ by \Cref{asymptotics}, there is a neighborhood $\widetilde{\mathcal{N}}$ of infinity such that $\widetilde{\mathcal{N}}\cap\critDu=\emptyset$. Also $\crit\subseteq\critDu$ as critical points of $u$ are global minima of $\vert Du\vert^{2}$.
\end{corollary}

\section{The divergence inequality} \label{section divergence inequality}
In this section, we will prove the geometric differential inequality \eqref{RobDivId5} and \Cref{divId Thm}. Our proof uses similar ideas as in the proof of \cite[Theorem 1.3]{Cederbaum.b} and of \cite{Robinson.1977} but also contains new ideas such as the application of the refined Kato inequality~\eqref{kato}. The rigidity part of the proof is new. 
\begin{proof}[Proof of \Cref{divId Thm}]
All claims in \Cref{divId Thm} are restricted to 
\begin{equation*}
V\definedas\R^{n}\setminus\left(\overline{\Omega}\cup\crit\right),
\end{equation*}
hence we only work on the open set $V\subset\R^{n}$. This allows us to apply \eqref{p.14 id 1}, \eqref{p.14 id 2}, and \eqref{norm rel hess grad norm} for $\beta\geq0$. For fixed $c,d\in\R$, we consider the smooth vector field
\begin{align*}
Z\definedas F(u)D|Du|^{\beta}+G(u)|Du|^{\beta} Du
\end{align*} 
on $V$, with $F$ and $G$ as in \eqref{solution F}, \eqref{solution G}, respectively. Exploiting that $u$ is harmonic, we deduce from \eqref{p.14 id 1}, \eqref{p.14 id 2}, and \eqref{norm rel hess grad norm} that
\begin{align*}
\diver Z
=&\;\frac{\beta}{2}\left(F'(u)+G(u)\right)|Du|^{\beta-2}\delta(D|Du|^2,Du)\\
&+\frac{\beta}{2} F(u)|Du|^{\beta-2}\left[\Delta|Du|^2+2(\beta-2)|D|Du||^2\right]+G'(u)|Du|^{\beta+2}
\end{align*}
on $V$. When $\beta=0$, this simplifies to
\begin{align*}
\diver Z=G'(u)|Du|^{2}=0
\end{align*}
on $V$ by \eqref{G'} which proves \eqref{RobDivId5} for $\beta=0$. Now let $\beta>0$. We apply the Bochner formula~\eqref{eq:Bochner} to get
\begin{align*}
|Du|^2\,\diver Z=\frac{\beta}{2}|Du|^{\beta}&\;\biggl[\left(F'(u)+G(u)\right)\,\delta(D|Du|^2,Du)\\
&+2F(u)\!\left(|D^2u|^2+(\beta-2)|D|Du||^2\right)+\frac{2G'(u)}{\beta}|Du|^{4}\biggr]
\end{align*}
on $V$. Note that $F(u)\geq0$ on $V$ for $c,d\in\R$ with $c+d\geq0$, $d\geq0$ by its definition \eqref{solution F} as $0<u<1$ on $V$. We can hence apply the refined Kato inequality \eqref{kato} to arrive at 
\begin{align*}
|Du|^2\,\diver Z
\geq\frac{\beta}{2}|Du|^{\beta}&\;\biggl[(F'(u)+G(u))\,\delta(D|Du|^2,Du)\\
&+\frac{2a_{\beta}F(u)}{\beta}|D|Du||^2+\frac{2G'(u)}{\beta}|Du|^{4}\biggl]
\end{align*}
on $V$, where we have used the definition of $a_{\beta}$ from \eqref{abeta}. Taking advantage of \Cref{ODEs}, we obtain \eqref{RobDivId5} by algebraic manipulations. Moreover, $a_{\beta}\geq0$ for $\beta\geq \frac{n-2}{n-1}$ which gives \eqref{divPos} on $V$, recalling that $F(u)\geq0$ on $V$.

Finally, we show the rigidity claim. First, if $\Omega$ is a round ball, one computes that equality holds in \eqref{divPos} by \eqref{radial u}. On the other hand, if equality holds in \eqref{divPos}, the right hand side of \eqref{RobDivId5} has to vanish as well and in particular we must have equality in the refined Kato inequality \eqref{kato}. Hence by \Cref{kato prop} we know that
\begin{align*}
D^{2}u(X,Y)&=\frac{\operatorname{tr}_{T}\left(D^{2} u\vert_{TW\times TW}\right)}{n-1}\,\delta(X,Y)
\end{align*}
holds on $V$ for all regular level sets $W=\{u=u_{0}\}\cap V$ of $u$ in $V$ and all vector fields $X,Y\in\Gamma(TW)$. By \Cref{Properties}, this shows that all regular (pieces of) level sets of $u$ are totally umbilic. As $\partial\Omega$ is a regular level set of $u$ and thus also has a (one-sided) tubular neighborhood in $\R^{n}\setminus\Omega$ which is entirely contained in $V$, it follows by continuity that $\partial\Omega$ is totally umbilic in $(\R^{n},\delta)$. As it is also connected and closed (i.e., compact without boundary), it is necessarily a round sphere and thus $\Omega$ is a round ball as claimed.
\end{proof}
	
\section{The parametric geometric inequality}\label{section geometric results}
In this section, we will prove \Cref{geom inequ thm} as an application of the divergence theorem to the vector field $Z$ from the proof of \Cref{divId Thm}. This is rather straightforward if $u$ has no critical points. Most of the work goes into handling $Z$ and the divergence theorem near the critical points of $u$, see \Cref{lem:regZ}. Some of the analytic arguments are inspired by those in \cite[Proposition 3.1, Lemma 3.4]{Agostiniani.2020}, but we add new twists to deal with additional difficulties, in particular to handle the different asymptotics\footnote{In \cite{Agostiniani.2020}, the analysis of some vector fields and the divergence theorem is performed after a conformal transformation into an asymptotically cylindrical picture.}, the singularity\footnote{The singularity of the integrand of the corresponding function $F$ is milder in \cite{Agostiniani.2020} due to the fact that they work in a conformal picture.} of the integrand of the function $\zeta$ in the proof of \Cref{lem:regZ}, and the case $\beta=\frac{n-2}{n-1}$ which requires more subtle handling in our setting. Also, our approach applies to general domains and not only to domains of the type $\{u_{0}<u<u_{1}\}$, see \Cref{section discussion} for more details. The divergence theorem for $Z$ will then be used to establish  \Cref{geom inequ thm} via \Cref{prop int id Rob} at the end of this section.

We recall from the proof of \Cref{divId Thm} that the vector field
\begin{align}\label{def:Z}
Z= F(u)D|Du|^{\beta}+G(u)|Du|^{\beta} Du
\end{align} 
was defined on the open set
\begin{equation}\label{def:V}
V=\R^{n}\setminus\left(\overline{\Omega}\cup\crit\right).
\end{equation}
 Recall furthermore from the proof of \Cref{divId Thm} that
\begin{align}
\begin{split}\label{eq:divZ}
\diver Z
=&\;F(u)|Du|^{\beta-4}\left\lbrace a_{\beta}\left\vert \frac{2(n-1)\vert Du\vert^{2}}{(n-2)u}Du-D\vert Du\vert^{2}\right\vert^{2}\right.\\
&\quad\quad\quad\quad\quad\quad\quad +\left.\beta\vert Du\vert^{2}\left(\vert D^{2}u\vert^{2}-\frac{n}{n-1}\vert D\vert Du\vert\vert^{2}\right)\right\rbrace
\end{split}
\end{align}
holds on $V$, where we have applied \Cref{ODEs}. In particular, the vector field $Z_{R_{\Omega}}$ corresponding to $\Omega=B_{R_{\Omega}}(0)$, with $R_{\Omega}$ given by \eqref{eq:asymptoticsR}, is smooth on $V_{R_{\Omega}}=\R^{n}\setminus\overline{B_{R_{\Omega}}(0)}$ as $u_{R_{\Omega},0}$ has no critical points, and is continuous with continuous derivatives up to $\mathbb{S}^{n-1}_{R_{\Omega}}(0)$. It satisfies $\diver Z_{R_{\Omega}}=0$ on $\R^{n}\setminus B_{R_{\Omega}}(0)$. We will rely on the following lemma to prove \Cref{geom inequ thm}.

\begin{lemma}[Extending $Z$ and $\nu$ to $\R^{n}\setminus\Omega$]\label{lem:regZ}
Let $\beta\geq\frac{n-2}{n-1}$, $c,d\in\R$. Then $Z$ and $\diver Z$ extend continuously to $\partial\Omega$. Extending $Z$, $\diver Z$, and $\nu=-\frac{Du}{\vert Du\vert}$ to $\crit$ by $+\infty$ makes them Lebesgue-measurable functions on $\R^{n}\setminus\Omega$ such that
\begin{enumerate}
\item $\diver Z\in L^{1}(\R^{n}\setminus\Omega)$,
\item $\delta(Z,\nu)\in L^{1}(\{u=u_{0}\}; d\sigma)$ for any regular level set $\{u=u_{0}\}$ of $u$, with
\begin{align}\label{asy:deltaZnu}
\lim_{u_{0}\to0+}\int_{\{u=u_{0}\}}\delta(Z,\nu)\,d\sigma&=\lim_{u_{0}\to0+}\int_{\{u=u_{0}\}}\delta\left(Z_{R_{\Omega}},\frac{x}{\vert x\vert}\right)\,d\sigma,
\end{align}
\item and the divergence theorem
\begin{align}\label{divthm}
\int_{U} \diver Z\,d\mu&=\int_{\partial U} \delta(Z,\eta)\,d\sigma
\end{align}
holds on any bounded domain $U\subseteq \R^{n}\setminus\Omega$ with smooth boundary $\partial U$ satisfying $\partial U\cap\crit=\emptyset$. Here, $\eta$ denotes the unit normal to $\partial U$ pointing out of $U$ and $d\sigma$ denotes the area measure induced on $\partial U$.
\end{enumerate}
\end{lemma}
\begin{proof}
Clearly, $Z$, $\diver Z$, and $\nu$ extend continuously to $\partial\Omega$ (recalling that $\partial\Omega$ is a regular level set of $u$). From \Cref{rem:infty}, we know that $Z$, $\diver Z$, and $\nu$ are smooth on some neighborhood $\mathcal{N}$ of infinity, so that $\mathcal{N}\subseteq V$. Note that \Cref{limits of FGH}, \Cref{asymptotics}, and \eqref{eq:asymptoticsR} assert that all quantities on the right hand sides of \eqref{def:Z} and \eqref{eq:divZ} asymptote to the corresponding quantities for $u_{R_{\Omega},0}$. Thus
\begin{align*}
Z&=Z_{R_{\Omega}}+o\left(\vert x\vert^{-(n-1)}\right),\\
\diver Z&=\diver Z_{R_{\Omega}}+o\left(\vert x\vert^{-n}\right)=o\left(\vert x\vert^{-n}\right)
\end{align*}
as $\vert x\vert\to\infty$ because $\diver Z_{R_{\Omega}}=0$ on $\R^{n}\setminus B_{R_{\Omega}}(0)$. This proves that $\diver Z\in L^{1}(\mathcal{N})$. Similarly, it implies that
\begin{align*}
\delta(Z,\nu)&=-\delta\left(Z_{R_{\Omega}},\frac{Du_{R_{\Sigma},0}}{\vert Du_{R_{\Sigma},0}\vert}\right)+o\left(\vert x\vert^{-(n-1)}\right)\\
&=\delta\left(Z_{R_{\Omega}},\frac{x}{\vert x\vert}\right)+o\left(\vert x\vert^{-(n-1)}\right)
\end{align*}
as $\vert x\vert\to\infty$ on the level sets of $u$ contained in $\mathcal{N}$. In particular, we have asserted that $\delta(Z,\nu)\in L^{1}(\{u=u_{0}\}; d\sigma)$ for all level sets $\{u=u_{0}\}\subset\mathcal{N}$. From \eqref{asydsigma}, we thus directly deduce Claim 2 of \Cref{lem:regZ}. 

If $\crit=\emptyset$, Claims 1 and 3 are obvious, by the above asymptotic assertions and the divergence theorem. Let us hence assume that $\crit\neq\emptyset$. As $\crit$ is a set of vanishing Lebesgue measure by \Cref{rem:sizecrit}, $Z$, $\diver Z$, and $\nu$ are Lebesgue-measurable on $\R^{n}\setminus\Omega$. To study the integrability claims near $\crit$, we aim to apply the monotone convergence theorem. However, as we will see, the arguments we will give only apply when $\beta>\frac{n-2}{n-1}$. We will hence handle the threshold case $\beta=\frac{n-2}{n-1}$ separately at the end of this proof. To apply the monotone convergence theorem for $\beta>\frac{n-2}{n-1}$, we consider a smooth cut-off function $\xi\colon [0,\infty)\to [0,1]$ that satisfies 
\begin{align*}
\begin{cases}
\xi(t)=0 &\text{if } t\leq \frac{1}{2},\\
\xi(t)=1&\text{if } t\geq \frac{3}{2},\\
0<\dot{\xi}(t)<2 &\text{if } \frac{1}{2}<t<\frac{3}{2}.
\end{cases}
\end{align*}
For $\varepsilon>0$, we define $\xi_{\varepsilon}\colon [0,\infty)\to [0,1]$ by setting $\xi_{\varepsilon}(t)\definedas \xi(\frac{t}{\varepsilon})$ and observe that
\begin{align*}
\begin{cases}
\xi_{\varepsilon}(t)=0 &\text{if } t\leq \frac{\varepsilon}{2},\\
\xi_{\varepsilon}(t)=1&\text{if } t\geq \frac{3\varepsilon}{2},\\
0<\dot{\xi}_{\varepsilon}(t)<\frac{2}{\varepsilon} &\text{if } \frac{\varepsilon}{2}<t<\frac{3\varepsilon}{2},\\
\xi_{\varepsilon_0}\leq \xi_{\varepsilon_1} &\text{if } 0<\varepsilon_{1}<\varepsilon_{0},\\
\xi_{\varepsilon}\to1&\text{as }\varepsilon\to0.
\end{cases}
\end{align*}
Furthermore, for every $\varepsilon>0$, we set 
\begin{equation*}
W_{\varepsilon}\definedas\{|Du|^2<\varepsilon\}\supseteq \crit.
\end{equation*}
Observe that $W_{\varepsilon}\subseteq\R^{n}\setminus\Omega$ is open for all $\varepsilon>0$. Moreover, for suitably small $\varepsilon>0$, we claim that $W_{\varepsilon}$ has only finitely many connected components, one being a neighborhood of infinity, the other ones being connected neighborhoods of (one or several) connected components of $\crit$ (see \Cref{rem:sizecrit}): First note that by \Cref{limits of FGH}, for each fixed $\varepsilon>0$ there exists a connected neighborhood of infinity which is contained in $W_{\varepsilon}$. Now suppose towards a contradiction that for a sequence $\{\varepsilon_{k}\}_{k\in\N}$ with $\varepsilon_{k}>0$ for all $k\in\N$ and such that $\varepsilon_{k}\to0$ as $k\to\infty$, there were a non-empty connected component $C_{k}$ of $W_{\varepsilon_{k}}$ such that $\bigcup_{k\in\N}C_{k}$ is bounded and $C_{k}\cap\crit=\emptyset$ for all $k\in\N$. Then there must be a constant $D>0$ and a
sequence $x_{k}\in C_{k}$ such that $\vert x_{k}\vert\leq D$ for all $k\in\N$. As $\overline{B_{D}(0)}$ is compact, there must be a subsequence (denoted without additional subscript for simplicity) with $x_{k}\to x_{*}$ as $k\to\infty$, and $x_{*}\in\crit$ by continuity of $\vert Du\vert^{2}$ and definition of $W_{\varepsilon_{k}}$. In particular, we have $x_{*}\in W_{\varepsilon_{k}}$ for all $k\in\N$. On the other hand, by assumption we have that $x_{*}\notin C_{k}$ for all $k\in\N$. As $\overline{W_{\varepsilon_{k+1}}}\Subset W_{\varepsilon_{k}}$ for all $k\in\N$, there is a uniform constant $E>0$ such that $\delta(x,y)\geq E$ for all $x,y\in W_{\varepsilon_{k}}$ lying in different components of $W_{\varepsilon_{k}}$ and all $k\geq 2$, leading to a contradiction. Hence, for suitably small $\varepsilon>0$, the connected components of $W_{\varepsilon}$ are neighborhoods of (one or more) connected components of $\crit$ plus one which is a neighborhood of infinity and no others. By compactness of $\crit$ (see \Cref{rem:sizecrit}), there can only be finitely many components of $W_{\varepsilon}$ for suitably small $\varepsilon>0$. Moreover, as $\crit\cap\partial\Omega=\emptyset$, $W_{\varepsilon}\cap\partial\Omega=\emptyset$ for suitably small $\varepsilon>0$.

Next, by definition, the boundary $\partial W_{\varepsilon}=\{\vert Du\vert^{2}=\varepsilon\}$ is closed and satisfies $\partial W_{\varepsilon}\cap\crit=\emptyset$. By \Cref{coro:Dusquared}, we know that $\vert Du\vert^{2}(\critDu)$ is discrete and that $\crit\subseteq\critDu$ and thus $0\in\vert Du\vert^{2}(\critDu)$. Hence there exists a threshold $\delta>0$ such that 
\begin{align*}
\vert Du\vert^{2}(x)\geq\delta
\end{align*}
for all $x\in \critDu\setminus\crit$. Then for all $0<\varepsilon<\delta$, the implicit function theorem applied to the smooth function $\vert Du\vert^{2}$ implies that $\partial W_{\varepsilon}$ is a smooth hypersurface with multiple but finitely many components (by boundedness of all components of $W_{\varepsilon}$ except the neighborhood of infinity). 

Now, using $\xi_{\varepsilon}$ for any $0<\varepsilon<\delta$, we cut off $\vert Du\vert^{2}$ near $\crit$, i.e., we study the function $\Theta_{\varepsilon}\colon\R^n\setminus \Omega\to[0,1]$ given by
\begin{align*}
\Theta_{\varepsilon}\definedas\xi_{\varepsilon}\circ |Du|^2,
\end{align*}
with $\supp \Theta_{\varepsilon}\subseteq \overline{W_{\frac{3\varepsilon}{2}}}\setminus W_{\frac{\varepsilon}{2}}$. Now let $\{\varepsilon_{k}\}_{k\in\N}$ be a strictly decreasing sequence of $\varepsilon_{k}>0$ satisfying
$\frac{3\varepsilon_{k}}{2}<\delta$ for all $k\in\N$ and $\varepsilon_{k}\to0$ as $k\to\infty$. With this choice of $\{\varepsilon_{k}\}_{k\in\N}$, $\{\Theta_{\varepsilon_{k}}\}_{k\in\N}\subset L^{1}(\R^{n}\setminus\Omega)$ is an increasing sequence, and we have $\Theta_{\varepsilon_{k}}\to1$ pointwise on $\R^{n}\setminus\Omega$ as $k\to\infty$. Now let $U$ be as in the statement of \Cref{lem:regZ}. If $U\cap\crit=\emptyset$, Claim 3 of \Cref{lem:regZ} holds on $U$ by the divergence theorem, while Claim 1 follows directly when $U\cap\crit=\emptyset$ for all $U$ as in Claim 3 because such $U$ excise $\R^{n}\setminus\Omega$. Hence assume that $U\cap\crit\neq\emptyset$. We aim at applying the divergence theorem to $\Theta_{\varepsilon_{k}}Z$ on $U$ and then take the limit as $k\to\infty$. To understand the volume integral, we use \eqref{def:Z}, \eqref{eq:divZ}, and \eqref{p.14 id 1} to find
\begin{align*}
&\diver(\Theta_{\varepsilon_{k}}Z)\\
&=\underbrace{\left(\dot{\xi}_{\varepsilon_{k}}\circ|Du|^2\right)\left[\frac{\beta}{2}F(u)|Du|^{\beta-2}|D|Du|^2|^2+G(u)|Du|^{\beta} \delta(D|Du|^2,Du)\right]}_{=:\, \mathcal{A}_{k}}+\underbrace{\vphantom{\left[\frac{\beta}{2}\right]}\Theta_{\varepsilon_{k}} \diver Z}_{=:\, \mathcal{B}_{k}}
\end{align*}
on $U$ for all $k\in\N$. Let us first discuss $\mathcal{B}_{k}$. Note that as $\Theta_{\varepsilon_{k}}$ vanishes near $\crit$, $\mathcal{B}_{k}\in L^{1}(U)$ for all $k\in\N$. Hence, by the monotone convergence theorem exploiting that $\diver Z\geq0$ almost everywhere on $U$ by \Cref{divId Thm} and as $\crit$ has Lebesgue measure zero, we find that
\begin{equation*}
\int_{U}\mathcal{B}_{k}\,d\mu=\int_{U}\Theta_{\varepsilon_{k}} \diver( Z)\, d\mu\to \int_{U} \diver( Z)\, d\mu\in\R\cup\{\infty\}
\end{equation*}
as $k\to\infty$. For $\mathcal{A}_{k}$, note that as $\xi_{\varepsilon_{k}}$ vanishes near $\crit$ we have $\mathcal{A}_{k}\in L^{1}(U)$ for all $k\in\N$. On the other hand, note that $\supp\mathcal{A}_{k}\subseteq\overline{W_{\frac{3\varepsilon_{k}}{2}}}\setminus W_{\frac{\varepsilon_{k}}{2}}$ for each $k\in\N$. Now recall that $F(u)\geq0$ in the setting of \Cref{lem:regZ} and that $0<\dot{\xi}_{\varepsilon_{k}}<\frac{2}{\varepsilon_{k}}$ for all $k\in\N$. Moreover, as $F(u)$ and $\vert G(u)\vert$ are continuous and $\vert Du\vert^{2}$ is smooth on $\R^{n}\setminus\Omega$, we know that the maps $s\mapsto\int_{U\cap\partial W_{s}}\!\!F(u)\,d\sigma$ and $s\mapsto\int_{U\cap\partial W_{s}}\!\vert G(u)\vert\,d\sigma$ for $s\in[\frac{\varepsilon_{k}}{2},\frac{3\varepsilon_{k}}{2}]$ are non-negative and Lebesgue-integrable on $[\tfrac{\varepsilon_{k}}{2},\tfrac{3\varepsilon_{k}}{2}]$ for all $k\in\N$. Moreover, $\vert Du\vert^{2}$ is Lipschitz continuous on $U$ as it is smooth on $\R^{n}\setminus\Omega$ and as $U\subset\R^{n}\setminus\Omega$ is bounded. Using the Cauchy--Schwarz inequality, the coarea formula (see e.g. \cite[Theorem 5]{Evans.2022}), and the mean value theorem for integrals, we compute
\begin{align*}
&\int_{U}\left\vert\mathcal{A}_{k}\right\vert d\mu\\
&\leq\!\! \int_{U\cap \left(\overline{W_{\frac{3\varepsilon_{k}}{2}}}\setminus W_{\frac{\varepsilon_{k}}{2}}\right)}\!\!\!\!\!\!\!\!\!\!\!\!\!\!\!\!\left(\dot{\xi}_{\varepsilon_{k}}\circ|Du|^2\right)\!\!\left[\frac{\beta}{2}F(u)|Du|^{\beta-2}|D|Du|^2|+\vert G(u)\vert |Du|^{\beta+1} \right] \vert D|Du|^2\vert\,d\mu\\
&=\int_{\frac{\varepsilon}{2}}^{\frac{3\varepsilon}{2}} \left(\,\,\int_{U\cap\partial W_s}\!\!\!\!\left(\dot{\xi}_{\varepsilon_{k}}\circ|Du|^2\right)\left[\frac{\beta}{2}F(u)|Du|^{\beta-2}|D|Du|^2|+|G(u)||Du|^{\beta+1}\right]d\sigma\!\right)\!ds
\end{align*}

\begin{align*}
&=\frac{\beta}{2}\int_{\frac{\varepsilon_{k}}{2}}^{\frac{3\varepsilon}{2}} \left(\dot{\xi}_{\varepsilon_{k}}(s)s^{\frac{\beta-2}{2}}\!\!\!\!\int_{U\cap\partial W_s}\!\!\!\!F(u)|D|Du|^2|\,d\sigma\!\right)\!ds+\int_{\frac{\varepsilon_{k}}{2}}^{\frac{3\varepsilon_{k}}{2}} \left(\dot{\xi}_{\varepsilon_{k}}(s)s^{\frac{\beta+1}{2}}\!\!\!\!\int_{U\cap\partial W_s}\!\!\!\!|G(u)|\,d\sigma\!\right)\!ds\\
&\leq \frac{\beta}{\varepsilon_{k}}\int_{\frac{\varepsilon_{k}}{2}}^{\frac{3\varepsilon_{k}}{2}} s^{\frac{\beta-2}{2}}\left(\,\,\int_{U\cap\partial W_s}\!\!\!\!F(u)|D|Du|^2|\,d\sigma\!\right)\! ds+\frac{2}{\varepsilon_{k}}\int_{\frac{\varepsilon}{2}}^{\frac{3\varepsilon_{k}}{2}} s^{\frac{\beta+1}{2}}\left(\,\,\int_{U\cap\partial W_s}\!\!\!\!|G(u)|\,d\sigma\!\right)\! ds\\
&= \beta r_{k}^{\frac{\beta-2}{2}}\!\!\!\!\int_{U\cap\partial W_{r_{k}}}\!\!\!\!F(u)\,d\sigma\,\int_{U\cap\partial W_{r_{k}}}\!\!\!\!|D|Du|^2|\,d\sigma+ 2r_{k}^{\frac{\beta+1}{2}}\!\!\!\!\int_{U\cap\partial W_{r_{k}}}\!\!\!\!|G(u)|\,d\sigma\\
&\leq \beta\,\max_{\overline{U}}F(u) \vert U\vert \underbrace{r_{k}^{\frac{\beta-2}{2}}\!\!\!\!\int_{\overline{U}\cap\partial W_{r_{k}}}\!\!\!\!|D|Du|^2|\,d\sigma}_{=:\,\mathcal{C}_{k}}+ 2\max_{\overline{U}}|G(u)| \vert U\vert\,\underbrace{\vphantom{\int_{\overline{U}\cap\partial W_{r_{k}}}\!\!\!\!|D|Du|^2|\,d\sigma}r_{k}^{\frac{\beta+1}{2}}}_{=:\,\mathcal{D}_{k}}
\end{align*}
for some $r_{k}\in(\frac{\varepsilon_{k}}{2},\frac{3\varepsilon_{k}}{2})$ and all $k\in\N$, by continuity of $|D|Du|^2|$, non-negativity and Lebesgue-integrability of $f,g$, and because $\overline{U}\subset \R^{n}\setminus\Omega$ is compact with $F(u)$ and $\vert G(u)\vert$ continuous on $\overline{U}$. Here, $\vert U\vert$ denotes the (finite, Euclidean) volume of $U$. Clearly, $\mathcal{D}_{k}\to0$ as $k\to\infty$ because $\beta>\frac{n-2}{n-1}$. We will now show that $\mathcal{C}_{k}\to0$ as $k\to\infty$, asserting by the above that 
\begin{equation}\label{eq:divconvergence}
\int_{U}\diver\left(\Theta_{\varepsilon_{k}}Z\right)d\mu\to\int_{U}\diver Z\,d\mu\in\R\cup\{\infty\}
\end{equation}
as $k\to\infty$. To analyze $\mathcal{C}_{k}$, we set 
\begin{equation*}
\rho_{U}\definedas\min\left\{\min_{\partial U}{\vert Du\vert^{2}},\delta\right\}>0
\end{equation*}
and choose $k_{0}=k_{0}(U,\Omega)\in\N$ such that $\frac{3\varepsilon_{k}}{2}<\rho_{U}$ for all $k\geq k_{0}$. This in particular implies $r_{k}<\rho_{U}$ for all $k\geq k_{0}$ for the above numbers $r_{k}$ arising from the mean value theorem for integrals. By the definition of $\delta$ and $\rho_{U}$, we find that $\partial U\cap \overline{W_{r}}=\emptyset$ and hence $\partial(U\cap W_{r})=U\cap\partial W_{r}$ for all $0<r<\rho_{U}$. Let us point out that intersecting with $U$ in particular excludes the component of $W_{r}$ which is a neighborhood of infinity. With this in mind, we study the auxiliary function $\zeta\colon(0,\rho_{U})\to\R$ defined by
\begin{equation*}
\zeta(r)\definedas\int_{U\cap W_{r}}\!\!\!\!|D|Du|^2|\,d\sigma
\end{equation*}
so that $\zeta\in L^{\infty}(0,\rho_{U})\subset L^{1}(0,\rho_{U})$ as $\vert Du\vert^{2}$ is continuous on $\R^{n}\setminus\Omega$ and $\overline{U}\subset \R^{n}\setminus\Omega$ is compact. Using that $\partial(U\cap W_{r})$ is a smooth hypersurface with finitely many components, applying the divergence theorem and the Bochner formula \eqref{eq:Bochner}, we get
\begin{align*}
\zeta(r)&=\int_{\partial(U\cap W_{r})}\!\!\!\!\!\!\delta\left(D|Du|^2,\frac{D|Du|^2}{|D|Du|^2|}\right)d\sigma=\int_{U\cap W_r}\!\!\!\!\diver\left( D|Du|^2\right)d\mu\\
&=\int_{U\cap W_r}\!\!\!\!\Delta |Du|^2\, d\mu=2\!\!\int_{U\cap W_r}\!\!\!\!|D^2u|^2\, d\mu
\end{align*} 
for all $0<r<\rho_{U}$ and thus by the coarea formula
\begin{equation*}
\zeta(\overline{r})-\zeta(r)=2\int_{r}^{\overline{r}} \left(\;\int_{U\cap \partial W_{s}}\!\!\!\! \frac{|D^2u|^2}{|D|Du|^2|} \,d\sigma\! \right)\!ds
\end{equation*}
for all $0<r\leq \overline{r}<\rho_{U}$ as $\vert D\vert Du\vert^{2}\vert$ is bounded from below by a positive constant on $\overline{U}\cap (W_{\overline{r}}\setminus W_{r})$ and thus $\frac{\vert D^{2}u\vert^{2}}{\vert D\vert Du\vert^{2}\vert}\in L^{\infty}(U\cap(W_{\overline{r}}\setminus W_{r}))\subset L^{1}(U\cap(W_{\overline{r}}\setminus W_{r}))$. For the same reason in combination with the fundamental theorem of calculus in the Sobolev space $W^{1,1}(\tau,\rho_{U})$, we have $\zeta\in W^{1,1}(\tau,\rho_{U})$ for any fixed $0<\tau<\rho_{U}$ with weak derivative  
\begin{align*}
\zeta'(r)=2\!\!\int_{U\cap\partial W_{r}}\!\!\frac{|D^2u|^2}{|D|Du|^2|}\,d\sigma
\end{align*}
for almost all $\tau<r<\rho_{U}$. In particular, it follows from the $1$-dimensional Sobolev embedding theorem that $\zeta$ is continuous on $(\tau,\rho_{U})$ for all $0<\tau<\rho_{U}$ and hence continuous on $(0,\rho_{U})$. Applying the refined Kato inequality \eqref{kato}, we deduce that
\begin{align*}
\zeta'(r)\geq \frac{2n}{n-1} \int_{U\cap\partial W_{r}}\!\!\frac{|D|Du||^2}{|D|Du|^2|}\,d\sigma=\frac{n}{2(n-1)}\frac{\zeta(r)}{r}
\end{align*}
for almost all $\tau<r<\rho_{U}$, using that $\vert D\vert Du\vert\vert=\frac{\vert D\vert Du\vert^{2}\vert}{2\vert Du\vert }$  and $Du\neq0$ hold on $U\cap\partial W_{r}$. As $0<\tau<\rho_{U}$ is arbitrary, this is equivalent to
\begin{equation*}
(\ln\circ\,\zeta)'(r)\geq \frac{n}{2(n-1)}\ln'(r)
\end{equation*}
for almost all $0<r<\rho_{U}$. Picking a fixed $0<R<\rho_{U}$ for which this inequality holds, this integrates to 
\begin{align*}
\zeta(r)\leq \frac{\zeta(R)}{R^{\frac{n}{2(n-1)}}}\,r^{\frac{n}{2(n-1)}}
\end{align*}
for all $0<r<R$ by continuity of $\zeta$. Hence
\begin{equation}\label{inequ:zeta}
0<r^{\frac{\beta-2}{2}}\zeta(r)\leq \frac{\zeta(R)}{R^{\frac{n}{2(n-1)}}}\,r^{\frac{1}{2}(\beta-\frac{n-2}{n-1})}
\end{equation}
holds for all $0<r<R$. As $\beta>\frac{n-2}{n-1}$, the exponent of $r$ on the right hand side of \eqref{inequ:zeta} is strictly positive so that $\mathcal{C}_{k}=r_{k}^{\frac{\beta-2}{2}}\zeta(r_{k})\to0$ as $k\to\infty$. This proves \eqref{eq:divconvergence} for $\beta>\frac{n-2}{n-1}$. Consider now the surface integral term
\begin{equation*}
\int_{\partial U}\delta(\Theta_{\varepsilon_{k}}Z,\eta)\,d\sigma.
\end{equation*}
As $\partial U\cap\crit=\emptyset$ holds by assumption, $Z$ is continuous on $\partial U$ and thus by compactness of $\partial U$ and by Lebesgue's dominated convergence theorem, we have
\begin{equation*}
\int_{\partial U}\delta(\Theta_{\varepsilon_{k}}Z,\eta)\,d\sigma\to \int_{\partial U}\delta(Z,\eta)\,d\sigma
\end{equation*}
as $k\to\infty$. Taken together and applying the divergence theorem to $\Theta_{\varepsilon_{k}}Z$ on $U$, this establishes both Claims 1 and 3 for $\beta>\frac{n-2}{n-1}$. 

To conclude Claims 1 and 3 for $\beta=\frac{n-2}{n-1}$, we consider a strictly decreasing sequence $\{\beta_{l}\}_{l\in\N}$ with $\beta_{l}>\frac{n-2}{n-1}$ and $\beta_l\to\frac{n-2}{n-1}$ as $l\to\infty$. Denoting $Z$ by $Z_{\beta}$ to be able to carefully consider \eqref{divthm} for different $\beta$, we have already asserted that
\begin{align*}
\int_{U} \diver Z_{\beta_{l}}\,d\mu&=\int_{\partial U} \delta(Z_{\beta_{l}},\eta)\,d\sigma
\end{align*}
for all $l\in\N$. Again using that $\partial U\cap\crit=\emptyset$, we know that $Z_{\beta_{l}}\to Z_{\frac{n-2}{n-1}}$ pointwise on $\partial U$ as $l\to\infty$. As $\{\vert Z_{\beta_{l}}\vert\}_{l\in\N}$ is uniformly bounded on $\partial U$ by compactness of $\partial U$ and continuity of all relevant quantities, we learn from Lebesgue's dominated convergence theorem that 
\begin{equation*}
\int_{\partial U} \delta(Z_{\beta_{l}},\eta)\,d\sigma\to\int_{\partial U} \delta(Z_{\frac{n-2}{n-1}},\eta)\,d\sigma
\end{equation*}
as $l\to\infty$. On the other hand, as $\crit$ has Lebesgue measure zero by \Cref{rem:sizecrit}, we know that  $\diver Z_{\beta_{l}}\to \diver Z_{\frac{n-2}{n-1}}$ on $U$ pointwise almost everywhere as $l\to\infty$. Splitting $U$ into $U\cap W_{1}$ and $U\setminus W_{1}$, Lebesgue's dominated convergence theorem tells us that 
\begin{align*}
\int_{U\setminus W_{1}}\!\!\diver Z_{\beta_{l}}\,d\mu \to \int_{U\setminus W_{1}}\!\!\diver Z_{\frac{n-2}{n-1}}\,d\mu\in\R
\end{align*}
as $l\to\infty$. On $U\cap W_{1}$, we rewrite \eqref{eq:divZ} as
\begin{align*}
\diver Z_{\beta_{l}}&=a_{\beta_{l}}\underbrace{F(u)|Du|^{\beta_{l}-4}\left\vert \frac{2(n-1)\vert Du\vert^{2}}{(n-2)u}Du-D\vert Du\vert^{2}\right\vert^{2}}_{=:\,\mathcal{E}_{l}}\\
&\quad+\beta_{l}\underbrace{F(u)\vert Du\vert^{\beta_{l}-2}\left(\vert D^{2}u\vert^{2}-\frac{n}{n-1}\vert D\vert Du\vert\vert^{2}\right)}_{=:\,\mathcal{F}_{l}}
\end{align*}
and note that $\{\mathcal{E}_{l}\}_{l\in\N},\{\mathcal{F}_{l}\}_{l\in\N}$ are non-negative sequences of Lebesgue-measurable functions on $U\cap W_{1}$ by \Cref{divId Thm}, because $F(u)\geq0$, and by the refined Kato inequality \eqref{kato}. Moreover, both $\{\mathcal{E}_{l}\}_{l\in\N},\{\mathcal{F}_{l}\}_{l\in\N}$ are monotonically increasing sequences on $U\cap W_{1}$ as
\begin{align*}
\frac{\partial \vert Du\vert^{\beta-b}}{\partial \beta}=\ln(\vert Du\vert)\vert Du\vert^{\beta-b}<0
\end{align*}
holds almost everywhere on $U\cap W_{1}$ and for all $b,\beta\in\R$. By the monotone convergence theorem, we hence find that
\begin{align*}
\int_{U\cap W_{1}}\!\!\diver Z_{\beta_{l}}\,d\mu=a_{\beta_{l}}\int_{U\cap W_{1}}\!\!\mathcal{E}_{l}\,d\mu+\beta_{l}\int_{U\cap W_{1}}\!\!\mathcal{F}_{l}\,d\mu\to \int_{U\cap W_{1}}\!\!\diver Z_{\frac{n-2}{n-1}}\,d\mu\in\R\cup\{\infty\}
\end{align*}
as $l\to\infty$. (Note however that we cannot conclude that the term involving $a_{\beta_{l}}$ vanishes in the limit as $\lim_{l\to\infty}\int_{U\cap W_{1}}\mathcal{E}_{l}\,d\mu$ may be infinite. This causes no issues as $a_{\beta_{l}}>0$ and $\mathcal{E}_{l},\mathcal{G}_{l}\geq0$.) This proves Claims 1 and 3 also for $\beta=\frac{n-2}{n-1}$.
\end{proof}

Next, we will deduce an integral identity that will be an important ingredient in the proof of \Cref{geom inequ thm}. A similar result appears in the proof of \cite[Corollary 3.5]{Agostiniani.2020}.
\begin{proposition}[Integral identity]\label{prop int id Rob}
Let $n\geq3$ and let $\Omega\subset \R^n$ be a bounded domain with smooth boundary and electrostatic potential $u$. Let $\beta\geq \frac{n-2}{n-1}$ and let $c,d\in\R$ be such that $c+d\geq0$, $d\geq0$. Consider the vector field $Z$ defined in \eqref{def:Z}, with $F$ and $G$ given by \eqref{solution F}, \eqref{solution G}, respectively. Let $0<u_{0}<u_{1}\leq 1$. Then
\begin{align}
\begin{split}\label{integral identity Rob}
\int_{\{u_{0}<u<u_{1}\}}\!\!\!\!\diver Z\,d\mu=&\!\int_{\{u=u_{1}\}}\!\!\left(\beta F(u) |Du|^{\beta}H+G(u)|Du|^{\beta+1}\right)d\sigma\\
&\quad-\! \int_{\{u=u_{0}\}}\!\!\left(\beta F(u) |Du|^{\beta}H+G(u)|Du|^{\beta+1}\right)d\sigma.
\end{split}
\end{align}
Here, for critical values $u_{*}$ of $u$, $\int_{\{u=u_{*}\}}\left(\beta F(u) |Du|^{\beta}H+G(u)|Du|^{\beta+1}\right)d\sigma$ is defined by the continuous extension of $\eta\mapsto \int_{\{u=\eta\}}\left(\beta F(u) |Du|^{\beta}H+G(u)|Du|^{\beta+1}\right)d\sigma$ to $u_{*}$. With this convention, the map $\eta\mapsto \int_{\{u=\eta\}}\left(\beta F(u) |Du|^{\beta}H+G(u)|Du|^{\beta+1}\right)d\sigma$ is continuous on $(0,1]$.
\end{proposition}
\begin{proof}
First, assume that $u_{0},u_{1}$ are regular values of $u$. From \Cref{lem:regZ}, we then know that
\begin{align*}
\int_{\{u_{0}<u<u_{1}\}}\!\!\!\!\diver Z\,d\mu&=-\!\int_{\{u=u_{1}\}}\!\!\delta(Z,\nu)\,d\sigma+\int_{\{u=u_{0}\}}\!\!\delta(Z,\nu)\,d\sigma.
\end{align*}
Applying \eqref{p.14 id 1}, \eqref{norm rel hess grad norm}, and \eqref{mean curv eucl}, we compute
\begin{align}\label{eq:Hidendity}
\delta(Z,\nu)&=-\beta F(u) \vert Du\vert^{\beta} H-G(u)|Du|^{\beta+1}
\end{align}
which implies the claim for non-critical values $u_{0},u_{1}$. If $u_{*}$ is a critical value of $u$, we know by \Cref{coro:finite} and the fact that $\partial\Omega$ is a regular level set of $u$ that there is a neighborhood $(u_{*}-2\varepsilon,u_{*}+2\varepsilon)$ of $u_{*}$ for some $\varepsilon>0$ which contains only regular values of $u$ (except $u_{*}$). We introduce 
\begin{equation*}
\Psi\colon(u_{*}-2\varepsilon,u_{*}+2\varepsilon)\setminus\{u_{*}\}\to\R\colon \eta\mapsto \int_{\{u=\eta\}}\!\!\left(\beta F(u)|Du|^{\beta}H+G(u)|Du|^{\beta+1}\right)d\sigma
\end{equation*}
and claim that $\Psi$ continuously extends to $u_{*}$. Applying this to $u_{*}=u_{0}$ and/or $u_{*}=u_{1}$ and exploiting that $\diver Z\in L^{1}(\R^{n}\setminus\Omega)$ by \Cref{lem:regZ}, this implies that \eqref{integral identity Rob} also applies to critical values $u_{0},u_{1}$ in the sense claimed in \Cref{prop int id Rob}.

Let us now verify the claim that $\Psi$ continuously extends to $u_{*}$. First, $\Psi$ is clearly well-defined away from $u_{*}$ by  \eqref{eq:Hidendity} and \Cref{lem:regZ}. Applying \eqref{integral identity Rob} to $u_{0}=\eta$, $u_{1}=u_{*}+\varepsilon$ and to $u_{0}=u_{*}-\varepsilon$, $u_{1}=\eta$, we find
\begin{equation*}
\Psi(\eta)=\Psi(u_{*}+\varepsilon)-\!\int_{\{\eta<u<u_{*}+\varepsilon\}}\!\!\!\!\diver Z\,d\mu=\Psi(u_{*}-\varepsilon)+\!\int_{\{u_{*}-\varepsilon<u<\eta\}}\!\!\!\!\diver Z\,d\mu
\end{equation*}
for $\eta\in(u_{*}-\varepsilon,u_{*}+\varepsilon)\setminus\{u_{*}\}$. As $\diver Z\geq0$ almost everywhere on $\R^{n}\setminus\Omega$ by \Cref{divId Thm} and \Cref{rem:sizecrit}, $\Psi$ is monotonically increasing on $(u_{*}-\varepsilon,u_{*}+\varepsilon)\setminus\{u_{*}\}$ and satisfies
\begin{equation*}
\Psi(u_{*}-\varepsilon)\geq \Psi(\eta) \geq \Psi(u_{*}+\varepsilon)
\end{equation*}
for all $\eta\in(u_{*}-\varepsilon,u_{*}+\varepsilon)\setminus\{u_{*}\}$. Thus the limits $\lim_{\eta\to u_{*}^{\pm}}\Psi(\eta)$ exist and are finite. Moreover, we see that
\begin{equation*}
\Psi(u_{*}+\varepsilon)-\Psi(u_{*}-\varepsilon)=\!\int_{\{u_{*}-\varepsilon<u<u_{*}+\varepsilon\}}\!\!\!\!\diver Z\,d\mu
\end{equation*}
holds for all suitably small $\varepsilon>0$, hence by absolute continuity of the Lebesgue integral and as $\diver Z\in L^{1}(\R^{n}\setminus\Omega)$, $\Psi$ can be continuously extended to $u_{*}$ by $\lim_{\eta\to u_{*}^{+}}\Psi(\eta)=\lim_{\eta\to u_{*}^{-}}\Psi(\eta)$. As this last identity also holds for non-critical values of $u$, the continuity claim readily follows.
\end{proof}

Finally, we apply \Cref{prop int id Rob} to show our main result.
\begin{proof}[Proof of \Cref{geom inequ thm}]
To show the claimed inequality \eqref{geom inequ}, we apply the integral identity \eqref{integral identity Rob} to $u_{1}=1$, i.e., we evaluate the first boundary integral at $\partial\Omega$. Also, we take $\lim_{\tau_0 \to \infty}u_{0}$, i.e., we evaluate the second boundary integral in  \eqref{integral identity Rob} at infinity which is permitted as $\diver Z\in L^{1}(\R^{n}\setminus\Omega)$ by \Cref{lem:regZ}. Appealing again to \Cref{prop int id Rob}, the definitions of $F$ and $G$ in \eqref{solution F}, \eqref{solution G}, and to the asymptotics established in \Cref{limits of FGH}, we find
\begin{align*}
\int_{\R^{n}\setminus\Omega}\!\!\diver Z\,d\mu
&= \beta (c+d)\int_{\partial\Omega} |Du|^{\beta}H\,d\sigma+\left(-\frac{n-1}{n-2}(c+d)\beta+d\right)\int_{\partial\Omega} |Du|^{\beta+1}\,d\sigma\\
&\quad-d(n-2)^{\beta+1}\capa^{\frac{n-2-\beta}{n-2}}\sarea.
\end{align*}
On the other hand, \Cref{divId Thm} together with \Cref{rem:sizecrit} establishes that $\diver Z\geq0$ almost everywhere on $\R^{n}\setminus\Omega$ which proves \eqref{geom inequ}. Finally, let us address the rigidity claim of \Cref{geom inequ thm} (unless $c=d=0$). When $\Omega$ is a round ball, we know from \Cref{divId Thm} that $\diver Z=0$ on $\R^{n}\setminus\Omega$ so that equality holds in \eqref{geom inequ}. On the other hand, if equality holds in \eqref{geom inequ}, we learn from the above identity that $\diver Z=0$ almost everywhere on $\R^{n}\setminus\Omega$. By \Cref{rem:sizecrit} and continuity of $\diver Z$ away from $\crit$, this gives that $\diver Z=0$ on $V=\R^{n}\setminus(\Omega\cup\crit)$ or in other words, equality holds in \eqref{divPos}. By  \Cref{divId Thm}, we know that this implies that $\Omega$ is a round ball.
\end{proof}

\section{Construction of monotone functionals and comparison to the monotonicity formula approach}\label{section discussion}
In \Cref{prop int id Rob}, we gave meaning to the surface integral corresponding to the divergence of the vector field $Z$ from \eqref{def:Z} also on critical level sets of the electrostatic potential $u$. We will now use this to define functionals $\mathcal{H}^{c,d}_{\beta}$ on the level sets of $u$ which are monotone and which we can then compare to the monotone functionals $\mathcal{F}_{\beta}$ introduced in \cite{Agostiniani.2020}. The relation between these two families of functionals turns out to be very similar to the relation of the monotone functionals studied in \cite{Cederbaum.b} and \cite{Agostiniani.2017} on the level sets of the static potential of an asymptotically flat static vacuum system in general relativity, see \cite{Cederbaum.b}.

Specifically, for a bounded domain $\Omega\subset\R^{n}$ with smooth boundary and electrostatic potential $u$ and for parameters $\beta\geq\frac{n-2}{n-1}$ and $c,d\in\R$ with $c+d\geq0$, $d\geq0$, we define the functional $\mathcal{H}^{c,d}_{\beta}\colon[1,\infty)\to\R$ by
\begin{align}
\mathcal{H}_\beta^{c,d}(\tau)\definedas \int_{\left\lbrace u=\frac{1}{\tau}\right\rbrace}\!\!\left( \beta F(u) |Du|^{\beta}H+G(u)|Du|^{\beta+1}\right)d\sigma,
\end{align}
using the convention established in \Cref{prop int id Rob} for critical values $\frac{1}{\tau}$ of $u$. As $\diver Z\geq0$ almost everywhere in $\R^{n}\setminus\Omega$ by \Cref{divId Thm} and \Cref{rem:sizecrit}, \Cref{prop int id Rob} implies that $\mathcal{H}^{c,d}_{\beta}$ is continuous and monotonically decreasing.
\begin{corollary}
Let $\Omega\subset\R^{n}$ be a bounded domain with smooth boundary and electrostatic potential $u$, and let $\beta\geq\frac{n-2}{n-1}$ and $c,d\in\R$ with $c+d\geq0$, $d\geq0$. Then $\mathcal{H}_{\beta}^{c,d}$ is continuous and monotonically decreasing.
\end{corollary}
Next, we recall the functionals $\mathcal{F}_{\beta}\colon[1,\infty)\to\R$ from \cite{Agostiniani.2020}, introduced in the exact same context. These are given by 
\begin{align}
 \mathcal{F}_{\beta}(\tau)\definedas\tau^{\frac{n-1}{n-2}\beta}\!\!\int_{\left\lbrace u=\frac{1}{\tau}\right\rbrace}\!\!|D u|^{\beta+1}\,d\sigma.
\end{align}
By \cite[Theorem 1.1]{Agostiniani.2020}, $\mathcal{F}_{\beta}$ is continuously differentiable with derivative given by
\begin{align}
\mathcal{F}_{\beta}'(\tau)=-\beta \tau^{\frac{n-1}{n-2}\beta-2}\!\int_{\left\lbrace u=\frac{1}{\tau}\right\rbrace}\!\!|Du|^{\beta}\left[H-\frac{n-1}{n-2}\,|D\log u| \right] d\sigma,
\end{align}
defined for critical values $\frac{1}{\tau}$ of $u$ by continuous extension as in \Cref{prop int id Rob}. Using the definitions of $F$, $G$ from \eqref{solution F}, \eqref{solution G}, one directly sees that
\begin{align}\label{eq:Hbeta}
\mathcal{H}_\beta^{c,d}&=-(c+d\tau)\mathcal{F}_{\beta}'+d\mathcal{F}_{\beta},\\\label{eq:Fbeta'}
\mathcal{F}_{\beta}'&=-\mathcal{H}_\beta^{1,0},\\\label{eq:Fbeta}
\mathcal{F}_{\beta}&=\mathcal{H}^{-1,1}_{\beta}+(1-\cdot)\mathcal{H}_{\beta}^{1,0}.
\end{align}
Now recall from the proof of \Cref{geom inequ thm} that $\mathcal{H}^{c,0}_{\beta}(\tau)\to0$ as $\tau\to\infty$, so that by monotonicity of $\mathcal{H}^{c,0}_{\beta}$ we have that $\mathcal{H}^{c,0}_{\beta}\geq0$ which gives $\mathcal{F}_{\beta}'\leq0$ via \eqref{eq:Fbeta'}, or in other words monotonicity of $\mathcal{F}_{\beta}$. Moreover, \eqref{eq:Fbeta'} combined with monotonicity of $\mathcal{H}^{c,0}_{\beta}$ easily gives that $\mathcal{F}_{\beta}$ is convex. We have hence also reproduced \cite[Theorem 1.1]{Agostiniani.2020}, once we note that the more extensive rigidity claims they make follow precisely by the same arguments we gave at the end of the proof of \Cref{divId Thm}, applied to the corresponding level set of $u$ instead of to $\partial\Omega$, appealing to the continuity of $\mathcal{H}^{c,d}_{\beta}$ in case that said level set intersects $\crit$.

\begin{corollary}
The Monotonicity-Rigidity Theorem \cite[Theorem 1.1]{Agostiniani.2020} follows from our approach (except for the claim about the formula for the second derivative of $\mathcal{F}_{\beta}$ which we don't address).
\end{corollary}

Looking at it the opposite way around, i.e., exploiting \cite[Theorem 1.1]{Agostiniani.2020} and defining $\mathcal{H}^{c,d}_{\beta}$ via \eqref{eq:Hbeta}, allows to conclude continuity of $\mathcal{H}^{c,d}_{\beta}$. In addition, \eqref{eq:Fbeta'} gives monotonicity of $\mathcal{H}_{\beta}^{1,0}$ and with this \eqref{eq:Fbeta} gives monotonicity of $\mathcal{H}^{-1,1}_{\beta}$ and thus of all $\mathcal{H}^{c,d}_{\beta}$. However, it is not clear how to proceed from there without applying \Cref{divId Thm}  (and parts of the analysis from the proof of \Cref{lem:regZ}) as we have used $\diver Z\geq0$ significantly in the proof of \Cref{prop int id Rob} via (the proof of) \Cref{lem:regZ}.

More abstractly, we have seen that the potential theoretic approach suggested by Agostiniani and Mazzieri~\cite{Agostiniani.2020}  lends itself not only to prove the Willmore inequality and other geometric inequalities via the very natural monotone functionals $\mathcal{F}_{\beta}$ they use, but also more directly via the divergence inequality and the divergence theorem. Indeed, when applying \Cref{prop int id Rob} in the proof of \Cref{geom inequ thm}, we could have instead directly appealed to \Cref{lem:regZ}, thereby completely avoiding to make sense of surface integrals over level sets for critical values of $u$ and needing to compute derivatives of functionals. Besides this difference of looking at the problem via the divergence theorem (in Euclidean space) rather than via monotone functionals over possibly critical level sets, we also do not need to switch into a conformal picture called the \emph{cylindrical ansatz} in \cite[(2.1),(2.2)]{Agostiniani.2020} (and apply the divergence theorem to two separate vector fields in the conformal setting \cite[Section 3]{Agostiniani.2020}). This avoids tedious computations and subtle analytic arguments\footnote{in particular asymptotic complications in the form of the growth condition needed in \cite[Corollary 3.6]{Agostiniani.2020} and the interior gradient estimate \cite[Proposition 2.3]{Agostiniani.2020}}  related to the conformal change and the asymptotically cylindrical nature of the conformal picture. Moreover, the rigidity argument becomes extremely simple in our case compared to the application of the splitting theorem \cite[Theorem 4.1-(i)]{Agostiniani.2015} and the necessity of a more careful handling of the threshold case $\beta=\frac{n-2}{n-1}$ in the analysis of the rigidity case. 

In summary, exploiting the potential theoretic setup by using the divergence inequality \eqref{RobDivId5} instead of using the beautiful cylindrical ansatz presents a significant simplification to proving the Willmore inequality. Moreover, our result suggests that the study of apparently less natural functionals ($\mathcal{H}^{c,d}_{\beta}$ rather than $\mathcal{F}_{\beta}$) can improve the understanding of the more natural functionals $\mathcal{F}_{\beta}$. This is very similar to what is found in \cite{Cederbaum.b} for black hole uniqueness and may lend itself to other problems studied with the help of monotone functionals in potential theoretic settings.\\[-4ex]

\bibliographystyle{amsplain}
\bibliography{Willmore}
	
\end{document}